\documentclass[12pt]{amsart} 

 \usepackage{caption}
 \usepackage{subcaption}

\usepackage[dvipsnames,usenames]{xcolor} 

\usepackage{amsfonts,latexsym,wasysym, mathrsfs, mathtools,hhline,color, bm}
\usepackage[all, cmtip]{xy}

\usepackage{url}

\definecolor{hot}{RGB}{65,105,225} 

\usepackage[pagebackref=true,colorlinks=true, linkcolor=hot ,  citecolor=hot, urlcolor=hot]{hyperref}

\usepackage{textcomp}
\usepackage{tipa}

\usepackage{graphicx,enumerate}
\usepackage{enumitem}
\usepackage[normalem]{ulem}

 \usepackage{amsthm}  
 \usepackage{amsmath} 
 \usepackage{amsfonts}  
 \usepackage{amssymb}
 \usepackage{amscd} 
 \usepackage[all]{xy} 

\usepackage{graphicx}
\usepackage{pdfsync}
\usepackage{url}
\usepackage{comment}
\usepackage{tikz}
\usetikzlibrary{arrows}
\usepackage[linesnumbered,ruled,vlined]{algorithm2e}
\usepackage{cleveref}

\theoremstyle{definition} 
\newtheorem{thm}{Theorem}[section]
\newtheorem{conj}[thm]{Conjecture}

\theoremstyle{definition}
\newtheorem{example}[thm]{Example}

\newtheorem{lem}[thm]{Lemma}
\newtheorem{prop}[thm]{Proposition}
\newtheorem{cor}[thm]{Corollary}
\newtheorem{defn}[thm]{Definition}

\newtheorem{question}[thm]{Question}
\newtheorem{rem}[thm]{Remark}
\numberwithin{equation}{section}

\newcommand{\Dpp}{\mathcal{D}}

\newcommand{\pmOne}{{\{-1,+1\}}} 
\newcommand{\cij}{{c}}
\newcommand{\ck}{{c}}

\newcommand{\RR}{\mathbb{R}}

\newcommand{\CC}{\mathbb{C}}

\newcommand{\sym}{{\mathrm{Sym}}}
\newcommand{\symnr}{{\mathrm{Sym}_n(\mathbb{R})}}

\DeclareMathOperator{\Diag}{Diag}
\DeclareMathOperator{\sign}{sign}
\DeclareMathOperator{\Trace}{Trace}

\DeclareMathOperator{\rank}{rank}

\title{Invariants of SDP exactness in quadratic programming}

\author[Lindberg]{Julia Lindberg}
\address{
Julia Lindberg \\
University of Texas-Austin \\ Austin, TX USA
}
\email{julia.lindberg@math.utexas.edu}
\urladdr{\url{https://sites.google.com/view/julialindberg/home}}

\author[Rodriguez]{Jose Israel Rodriguez}
\address{
Jose Israel Rodriguez\\
University of Wisconsin-Madison \\
 Madison, WI USA
} 
\email{jose@math.wisc.edu}
\urladdr{\url{https://sites.google.com/wisc.edu/jose/home}}

\begin{document}

\maketitle

\begin{abstract}
In this paper we study the Shor relaxation of quadratic programs by fixing a feasible set and considering the space of objective functions for which the Shor relaxation is exact. We first give conditions under which this region is invariant under the choice of generators defining the feasible set. We then describe this region when the feasible set is invariant under the action of a subgroup of 
the general linear group. 
We conclude by applying these results to quadratic binary programs. We give an explicit description of objective functions where the Shor relaxation is exact and use this knowledge to design an algorithm that produces candidate solutions for binary quadratic programs.
\end{abstract}

\section{Introduction}

Quadratically constrained quadratic programs (QCQPs) are a broad class of polynomial  optimization problems that seek to minimize a quadratic objective function subject to quadratic equality constraints. Specifically, these are problems of the form:
\begin{align}
\min_{x \in \mathbb{R}^n} \ f_0(x) \ \quad \text{subject to} \ \quad f_i(x) = 0, \ i \in [m] \label{eq:1}
\end{align}
where $f_0,\ldots,f_m \in \mathbb{R}[x_1,\ldots,x_n]_{\leq 2} :=\{ f\in \RR[x_1,\dots x_n] : \deg(f)\leq 2 \}$ and $[m]:=\{1,\ldots,m\}$. Using matrix notation to write $f_0(x) = x^TCx + 2 c^T x$ and $f_i(x) = x^T A_i x + 2 a_i^T x + \alpha_i$, $i \in [m]$,
where $C, A_i \in \text{Sym}_n(\mathbb{R})$ are symmetric $n \times n $ matrices, $c, a_i \in \mathbb{R}^n$ and $\alpha_i \in \mathbb{R}$.
We consider optimization problems of the form: 
\begin{align*}
    \min_{x \in \mathbb{R}^n} \ x^T C x + 2c^T x \quad \text{subject to} \quad x^T A_i x + 2a_i^T x + \alpha_i = 0, \ i \in [m]. \tag{QCQP} \label{QP}
\end{align*}

QCQPs have broad modelling power and have found applications in signal processing, combinatorial optimization, power systems engineering and more \cite{tan2001the,khabbazibasmenj2014generalized, poljak1995a, lee2011mixed, molzahn2019a, zhong2013dynamic, papaspiliotopoulos2017a}. 
In general, these problems are NP hard to solve \cite{vavasis1990quadratic} but a convex relaxation defined by a \textit{semidefinite program (SDP)}  gives an outer relaxation of \eqref{QP}. This relaxation, called the \textit{Shor relaxation} \cite{shor1987quadratic}, lifts the optimization variable $x \in \mathbb{R}^n$ to $\binom{n+1}{2}$-dimensional space by considering the optimization variable $X \in \text{Sym}_{n+1}(\mathbb{R})$. A detailed derivation of this relaxation is given in \cite{luo2010semidefinite} but we outline the idea below.

 Consider $x = (x_1,\ldots,x_n) \in \mathbb{R}^{n} $. The key observation is that for any $A \in \symnr$,
 $$x^T A x = \Trace(x^T A x) = \Trace(Axx^T).$$
Therefore, applying this technique to $x = (1,x_1,\ldots,x_n)$ and observing that in this case 
\[
X = x x^T = \begin{bmatrix}
    1 & x_1 & x_2  & \cdots & x_n \\
    x_1 & x_1^2 & x_1x_2  & \cdots & x_1 x_n \\
    \vdots & & \ddots  &  & \vdots \\
    x_{n-1} & x_1 x_{n-1} & x_2 x_{n-1}  & \cdots & x_{n-1}x_n \\
    x_n & x_1 x_n & x_2 x_n  & \cdots & x_n^2
\end{bmatrix},
\]
we define
\[ 
\mathcal{C} := \begin{bmatrix} 0 & c^T \\ c & C \end{bmatrix}, \quad \mathcal{A}_i = \begin{bmatrix} \alpha_i & a_i^T \\ a_i & A_i \end{bmatrix} \quad \text{and} \quad \mathcal{A}_0 = \begin{bmatrix} 1 & 0_{1 \times n} \\ 0_{n \times 1} & 0_{n \times n} \end{bmatrix},
\] 
to produce a mathematically equivalent formulation of \eqref{QP} as:
\begin{align*}
    \min_{X \succeq 0} \ \langle \mathcal{C},X \rangle \quad \text{subject to} \quad \langle \mathcal{A}_i,X \rangle &= 0, \ i \in [m] \\
    \langle \mathcal{A}_0,X \rangle &= 1 \\
    \text{rank}(X) &=1,
\end{align*}
where $\langle X, Y \rangle = \Trace(X \cdot Y)$ denotes the trace.

The key observation is that for any $A\in \sym_{n+1}(\mathbb{R})$ and rank one matrix $V\in \sym_{n+1}(\mathbb{R})$, 
there exists $v\in \mathbb{R}^{n+1}$ such that 
$V=vv^T$ and 
 \[v^T A v = \Trace(v^T A v) = \Trace(AV).
 \]
To apply this idea to our situation, we take $v=(1,x_1,\dots,x_n)$
and we define 
\[ 
\mathcal{C} := \begin{bmatrix} 0 & c^T \\ c & C \end{bmatrix}, \quad \mathcal{A}_i = \begin{bmatrix} \alpha_i & a_i^T \\ a_i & A_i \end{bmatrix} \quad \text{and} \quad \mathcal{A}_0 = \begin{bmatrix} 1 & 0_{1 \times n} \\ 0_{n \times 1} & 0_{n \times n} \end{bmatrix}.
\] 
Observing 
that in this case 
\[
X = v v^T = \begin{bmatrix}
    1 & x_1 & x_2  & \cdots & x_n \\
    x_1 & x_1^2 & x_1x_2  & \cdots & x_1 x_n \\
    \vdots & & \ddots  &  & \vdots \\
    x_{n-1} & x_1 x_{n-1} & x_2 x_{n-1}  & \cdots & x_{n-1}x_n \\
    x_n & x_1 x_n & x_2 x_n  & \cdots & x_n^2
\end{bmatrix} 
\]
we produce a mathematically equivalent formulation of \eqref{QP} as
\begin{align*}
    \min_{X \succeq 0} \ \langle \mathcal{C},X \rangle \quad \text{subject to} \quad \langle \mathcal{A}_i,X \rangle &= 0, \ i \in [m] \\
    \langle \mathcal{A}_0,X \rangle &= 1 \\
    \text{rank}(X) &=1,
\end{align*}
where $\langle X, Y \rangle = \Trace(X \cdot Y)$ denotes the trace.

Observe that our original optimization variables were $x_1,\ldots,x_n$ but in the relaxation, the optimization variable is the $(n+1) \times (n+1)$ matrix $X = xx^T$ where $x = (1,x_1,\ldots,x_n)$. The rank constraint is nonconvex, but it is the only nonconvex constraint. The Shor relaxation is then given by removing this nonconvex constraint. Specifically, it is defined as:
\begin{align*}
    \min_{X \succeq 0} \ \langle \mathcal{C},X \rangle \quad \text{subject to} \quad \langle \mathcal{A}_i,X \rangle &= 0, \ i \in [m] \\
    \langle \mathcal{A}_0,X \rangle &= 0. \tag{QCQP-Relax} \label{QP-Relax}
\end{align*}
If an optimal solution, $X^*$, to \eqref{QP-Relax} is rank $1$, the \emph{relaxation is exact}, as we can write $X^* = x^* x^*{^T}$ and $x^*$ is an optimal solution to \eqref{QP}.

The Shor relaxation gives the first iteration in what is known as the moment/sums of squares hierarchy \cite{parrilo2003semidefinite,lasserre2000global}. This hierarchy gives a series of SDP relaxations to \eqref{QP} that generically converges to its global optimum \cite{nie2014optimality}. The main issue is that the size of the SDP increases at an exponential rate. Consequently, for even small size problems it is not computationally tractable to go beyond the first iteration of this hierarchy. Therefore, a natural question is to ask when \eqref{QP-Relax} is exact.

\begin{question}\label{ques:sdp_exact}
What are necessary and sufficient conditions on $\mathcal{C}, \mathcal{A}_i$ for $i \in [m]$ such that \eqref{QP-Relax} is exact?
\end{question}

 Some partial answers have been given to \Cref{ques:sdp_exact}. One of the first results in this area showed that if $m = 1$ (i.e. there is a single equality constraint) then \eqref{QP-Relax} is exact \cite{flippo1996duality}. Later work extended this by considering \eqref{QP} with two quadratic constraints \cite{LOCATELLI2015126, ugur2009convex}.
Other work in this area considers when the feasible region of \eqref{QP} is defined by inequality constraints\footnote{Observe that the assumption in this paper of equality constraints can encompass inequalities by adding slack variables.}. For instance if all matrices $C, A_1,\ldots,A_n \succeq 0$ then \eqref{QP} is convex and \eqref{QP-Relax} is exact \cite{fujie1997semidefinite}. Another condition, namely if the off-diagonal entries of $C, A_i$, $i \in [m]$ are non-positive, then \eqref{QP-Relax} is exact \cite{kim2003exact}. Other work has considered when $C, A_i$ are sparse. In \cite{sojoudi2014exactness} the authors map the structure of \eqref{QP} to a generalized weighted graph and show that if each weight set is sign-definite and conditions on each cycle of the graph are satisfied, then the relaxation is exact. 
Other work has been able to bound the rank of the optimal solution to \eqref{QP-Relax} by using a pre-processing step done in polynomial time that is dependent on the data $\mathcal{C},\mathcal{A}_i$, $i \in [m]$ \cite{burer2020exact}. A recent line of work gives sufficient conditions on both objective value and convex hull exactness (the condition that the
convex hull of the epigraph of \eqref{QP} coincides with the epigraph of \eqref{QP-Relax}) \cite{wang2021geometric, wang2020on, wang2022tightness}.

For fixed $\mathcal{A}_i$, $i \in [m]$, we are interested in studying the space of objective functions where \eqref{QP-Relax} is exact. This direction of research is inspired by the work in \cite{cifuentes2017on,cifuentes2020the} where, using strict complementarity conditions, the authors define the \emph{SDP exact region} of \eqref{QP}. 
Their key definition is stated in terms of the Lagrangian
\[ 
\mathcal{L}:~\RR^m \times \RR^n \to \RR,\quad
\mathcal{L}(\lambda, x) = x^T C x + 2c^Tx - \sum_{i=1}^m \lambda_i (x^TA_ix + 2a_i^Tx + \alpha_i) \]
of the \eqref{QP}
and the Hessian of $\mathcal{L}$, 
\[
H:\RR^m\to \RR^{n\times n},\quad
H(\lambda) = C - \sum_{i=1}^m \lambda_i A_i. 
\]

We denote $F = (f_1,\ldots,f_m)$ as the list of polynomials defining the feasible region of \eqref{QP} and its \emph{real variety}
\[\mathcal{V}_{\RR}(F) := \{x \in \RR^n \ : \ f_1(x) = 0,\ldots,f_m(x) = 0\}. \]

\begin{defn} \cite[Def. 3.2]{cifuentes2020the}\label{def:sdpExact}
The \emph{SDP exact region} of \eqref{QP}, 
$\mathcal{R}_F$, is the set of objective functions $(C,\ck) \in (\sym_n(\mathbb{R}), \mathbb{R}^n)$ such that the Shor relaxation of \eqref{QP} is exact.
Specifically,
\[ \mathcal{R}_F = \{(C,\ck) \ : \ H(\lambda) \succ 0, \ c - \sum_{i=1}^m \lambda_i a_i + H(\lambda) x = 0 \ \text{ for some } x \in \mathcal{V}_\mathbb{R}(F), \lambda \in \mathbb{R}^m \}.  \]
\end{defn}

\begin{rem}
The SDP exact region of \eqref{QP} is defined in terms of the \emph{set of polynomials} whose common real zero set defines the feasible region of \eqref{QP}. It is \emph{not} defined in terms of the ideal generated by these polynomials nor by the variety of common solutions.  
 This subtle point will be discussed in detail in \Cref{sec:2}.

\end{rem}

Motivated by \Cref{ques:sdp_exact}, this paper builds on work in \cite{cifuentes2017on, cifuentes2020the} towards understanding the geometry 
of SDP exact regions
for various classes of quadratic programs. 
In \cite{cifuentes2017on} the authors show that the SDP exact region of the Euclidean distance problem is full dimensional. 
In \cite{cifuentes2020the} the authors study the algebraic boundary of this region and the degree of the hypersurface defining it. 

Our results address fundamental questions on invariants of SDP exact regions. 
In \Cref{sec:2} we give conditions under which the SDP exact region of an ideal is invariant under the choice of generators. 
Specifically, we determine when 
two different sets of generators of an ideal define the same SDP exact region.
In \Cref{sec:3} we consider when the feasible region of \eqref{QP} has symmetry. 
We conclude in \Cref{sec:4} by applying these results to quadratic binary programs and use our understanding of $\mathcal{R}_F$ to design an algorithm that gives candidate solutions to \eqref{QP}.

\section{Different generators of the same ideal}\label{sec:2}

We first wish to understand one of the most fundamental questions about $\mathcal{R}_F$, namely when do two sets of equations defining the same feasible region have the same SDP exact region? 

For a list of polynomials $F = ( f_1,\ldots,f_m )$, 
we denote its  \emph{ideal} as
\[
\mathcal{I}(F) = \{ h_1 f_1 + \ldots + h_m f_m \ : \  h_i \in \mathbb{R}[x_1,\ldots,x_n] \}
\]
and the \emph{complex variety} of $F$ as
\begin{align*}
    \mathcal{V}_\mathbb{C}(F) &= \{ x \in \mathbb{C}^n \ : \ f_1(x) = 0, \ldots, f_m(x) = 0 \}.
\end{align*}

By the  Ideal–Variety Correspondence, for any two sets of polynomials over an algebraically closed field, $F = (f_1,\ldots,f_m)$ and $G = (g_1,\ldots,g_N)$
the equality of ideals $\mathcal{I}(F) = \mathcal{I}(G)$ implies the equality of varieties $\mathcal{V}_{\CC}(F) = \mathcal{V}_{\CC}(G)$. 
The goal of this section is to establish conditions under which $\mathcal{I}(F) = \mathcal{I}(G) $ implies an equality of SDP exact regions $\mathcal{R}_F = \mathcal{R}_G$. 
We begin with an example that shows this is not always the case.

\begin{example}\label{ex:1}
Consider $F = \{ x^2 - 1, xy - 1 \}$ and $G = \{ xy - y^2, y^2 - 1 \}$. 
One can verify by polynomial division that each polynomial of $G$ is in the ideal $\mathcal{I}(F) $ and vice versa to conclude  $\mathcal{I}(F) = \mathcal{I}(G)$.
Using \Cref{def:sdpExact} we see 
\begin{align*}
    \mathcal{R}_F &= \{ (C,\ck) \ : \ \begin{pmatrix}
    -\ck_1+\ck_2+\ck_{22} & -\ck_2 - \ck_{22} \\ -\ck_2 - \cij_{22} & \cij_{22} 
    \end{pmatrix} \succ 0 \ \text{or} \ \begin{pmatrix}
    \ck_1-\ck_2+\cij_{22} & \ck_2 - \cij_{22} \\ \ck_2 - \cij_{22} & \cij_{22} 
    \end{pmatrix} \succ 0 \} \\
    \mathcal{R}_G &= \{ (C,\ck) \ : \ \begin{pmatrix}
    \cij_{11} & -\ck_1 - \cij_{11} \\ - \ck_1 - \cij_{11} & \ck_1 + \cij_{11} - \ck_2 
    \end{pmatrix} \succ 0 \ \text{or} \ \begin{pmatrix}
    \cij_{11} & \ck_1 - \cij_{11} \\ \ck_1 - \cij_{11} & -\ck_1 + \cij_{11} + \ck_2 
    \end{pmatrix} \succ 0 \}.
\end{align*}

A figure of $\mathcal{R}_F$ projected onto $\ck_{1},\ck_2,\cij_{22}$ is shown in \Cref{fig:sdpex}. While $\mathcal{I}(F) = \mathcal{I}(G)$, $\mathcal{R}_F \neq \mathcal{R}_G$ in this case. 
In other words, even though $F=0$ and $G=0$ have the same set of solutions, the SDP exact regions are different. 
For instance in $\mathcal{R}_F$ we require $c_{22} >0$ but in $\mathcal{R}_G$ there are no constraints on $c_{22}$. \hfill
$\diamond$

\begin{figure}[h!]
    \centering
    \includegraphics[width = 0.4\textwidth]{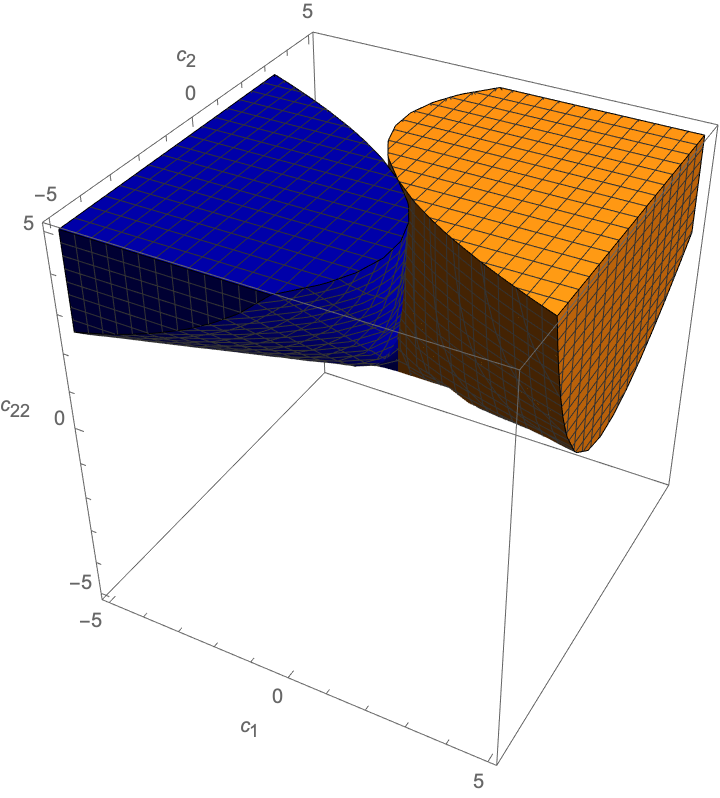}
    \caption{$\mathcal{R}_F$ from \Cref{ex:1}.}
    \label{fig:sdpex}
\end{figure}
\end{example}

\Cref{ex:1} shows that the SDP exact region of a variety is not always invariant under the choice of generators. This is important as it highlights that the choice of generators of the ideal defining the feasible region of \eqref{QP} can impact the computational efficiency of solving \eqref{QP}.
We would like to understand for which ideals the choice of generators is irrelevant. 
We start in the simplest case, namely when $\mathcal{V}_\mathbb{R}(F)$ is empty. 
This is an important case since determining feasibility 
is a natural starting point when attempting to solve any optimization problem. If \eqref{QP-Relax} is infeasible, then so is \eqref{QP} but the reverse implication is not always true. The next result shows that proving \eqref{QP} is infeasible is equivalent to showing $\mathcal{R}_F$ is empty.

\begin{prop}\label{prop:Rf_empty}
With the notation as in Definition~\ref{def:sdpExact},
$\mathcal{R}_F = \emptyset$ if and only if $\mathcal{V}_\mathbb{R}(F) = \emptyset$.
\end{prop}
\begin{proof}
$(\Rightarrow):$ We prove this by contrapositive. Assume $\mathcal{V}_\mathbb{R}(F) \neq \emptyset$, we want to show that $\mathcal{R}_F \neq \emptyset$. Let $x \in \mathcal{V}_\mathbb{R}(F)$ and fix $\lambda \in \mathbb{R}^m$ and $C \in \text{Sym}_n(\mathbb{R})$ such that $H(\lambda) \succ 0$. Observe that such a $C$ exists since taking $C = \alpha \cdot I$ where $\alpha$ is greater than the smallest eigenvalue of $\sum_{i=1}^m \lambda_i A_i$ works. Then choose $c \in \mathbb{R}^n$ such that $c - \sum_{i=1}^m \lambda_i a_i + H(\lambda)x = 0 $. \\
\hfill \\
$(\Leftarrow):$ Suppose $\mathcal{V}_\mathbb{R}(F) = \emptyset$. Then trivially by \Cref{def:sdpExact}, $\mathcal{R}_F = \emptyset$.
\end{proof}

The condition for $\mathcal{R}_F$ to be empty as outlined in \Cref{prop:Rf_empty} is a natural one. We now extend our analysis 
to the case when $\mathcal{V}_\mathbb{R}(F)$ is nonempty. 
Define the \emph{dimension} of a variety $\mathcal{V}_{\mathbb{C}}(F)$ to be the maximal number of general\footnote{By `general', we mean the coefficients of the linear equations defining the hyperplanes lie in a dense, Zariski open set. This restricts the non-generic behavior to a Zariski closed set of codimension at least one, that is necessarily of Lebesgue measure zero.} hyperplanes whose common intersection with $\mathcal{V}_\CC(F)$ is nonempty.

\begin{prop}\label{thm:sdp_under_linear}
Consider the sets of polynomials
$F=(f_1,\dots, f_m)$ and $G = (g_1,\ldots,g_N)$ in $\RR[x_1,\dots,x_n]_{\leq 2}$ 
with  $\dim ( \mathcal{V}_{\mathbb{C}}(F)) =n-m$.
If each polynomial in $G$ is a $\RR$-linear combination of $F$ and 
$\mathcal{I}(F) = \mathcal{I}(G)$, then 
 $\mathcal{R}_F = \mathcal{R}_G$.

\end{prop}
\begin{proof}
If $\mathcal{V}_{\mathbb{R}}(F) = \emptyset$, then by \Cref{prop:Rf_empty}, $\mathcal{R}_F = \mathcal{R}_G~=~\emptyset$. Therefore, we assume $\mathcal{V}_{\mathbb{R}}(F)~\neq~\emptyset$ and first consider the case $N = m$.

Let $g_t = x^T \hat{A}_t x + 2 \hat{a}_t^T x + \hat{\alpha}_t$ and $f_i = x^T A_i x + 2a_i^T x + \alpha_i$ for $i, t \in [m]$. 
Since each $g_t$ is a $\RR$-linear combination of $f_1,\dots,f_m$, there exists $\ell_{ti}\in \RR$ such that
$g_t =  \sum_{i=1}^m \ell_{ti} f_i$:
\[g_t = \sum_{i=1}^m \ell_{ti}f_i = \sum_{i=1}^m \ell_{ti} (x^T A_i x + 2a_i^T x + \alpha_i), \quad t \in [m].\] 

First we show $\mathcal{R}_F \subseteq \mathcal{R}_{G}$. Suppose $(C,c) \in \mathcal{R}_F$ corresponding to point $x \in \mathcal{V}_{\mathbb{R}}(F)$ and $\lambda \in \mathbb{R}^m$. We claim $(C,c) \in \mathcal{R}_{G}$ corresponding to point $x \in \mathcal{V}_{\mathbb{R}}(G)$ and $\hat{\lambda} \in \mathbb{R}^m$ where $\hat{\lambda} = L^{-T}\lambda$ and $L \in \mathbb{R}^{m \times m}$ is the matrix with $[L]_{ti} = \ell_{ti}$. Since $\mathcal{I}(F) = \mathcal{I}(G)$, and $\dim(\mathcal{V}_{\mathbb{C}}(F)) = n-m$, we have that $\rank(L) = \rank(L^{-1}) = m$, therefore, such a $\hat{\lambda}$ exists. Let $\hat{H}$ and $H$ be the Hessian of \eqref{QP} with feasible set defined by $G$ and $F$ respectively. Then,
\begin{align*}
    \hat{H}(\hat{\lambda}) &= \sum_{t=1}^m \hat{\lambda}_t \hat{A}_t \\ 
    &= \sum_{t=1}^m \hat{\lambda}_t \sum_{i=1}^m \ell_{ti} A_i \\
    &= \sum_{i=1}^m A_i \sum_{t=1}^m \hat{\lambda}_t \ell_{ti} \\
    &= \sum_{i=1}^m \lambda_i A_i = H(\lambda) \succ 0
\end{align*}
where the last line is by definition of $\hat{\lambda}$. Now consider the equality constraints in \Cref{def:sdpExact}:
\begin{align*}
c - \sum_{t=1}^m \hat{\lambda}_t \hat{a}_t + \hat{H}(\hat{\lambda})x &=  c - \sum_{t=1}^m \hat{\lambda}_t \sum_{i=1}^m \ell_{ti}a_i + H(\lambda) x \\
&= c - \sum_{i=1}^m a_i \sum_{t=1}^m \hat{\lambda}_t \ell_{ti} + H(\lambda) x \\
&= c - \sum_{i=1}^m \lambda_i a_i + H(\lambda ) x = 0.
\end{align*}
The same argument holds for showing $\mathcal{R}_{G} \subseteq \mathcal{R}_F$ by considering $\lambda = L^T \hat{\lambda}$.

Now consider when $N > m$. Since each polynomial in $G$ is a $\RR$-linear combination of $f_1,\ldots,f_m$, up to reordering, there exists a choice of $g_1,\ldots g_m$ such that $g_{m+1},\ldots,g_N$ are $\RR$-linear combinations of $g_1,\ldots,g_m$ and the $\RR$-linear combination is full rank.
\end{proof}

\Cref{thm:sdp_under_linear} tells us that for a polynomial system $F = 0$, $\mathcal{R}_F$ is invariant under the choice of generators of $\mathcal{I}(F)$ if for any other generators $G$ such that $\mathcal{I}(F) = \mathcal{I}(G)$, $G$ can be written as $G = L \cdot F$ for some $L \in \mathbb{R}^{N \times m}$ with rank $m$. 
Observe that for given polynomials $F$ and $G$, deciding whether $G = L \cdot F$ can be checked easily. 
A more difficult question is: Given a polynomial system $F = 0$, can \emph{every} polynomial system $G = 0$ such that $\mathcal{I}(F) = \mathcal{I}(G)$ be written as $G = L \cdot F$? 
We give a sufficient condition under which ideals defined by quadratic generators satisfy this assumption.

\begin{thm}\label{cor:sdp_welldefined}
Consider $F = ( f_1,\ldots,f_m )$ where  
$f_i\in \mathbb{R}[x_1,\ldots,x_n]_{\leq 2}$
for $i \in [m]$ and $m \leq n$.
Write $f_i = q_i + p_i$ where $q_i$ is the homogeneous degree two part of $f_i$ and $p_i$ is the degree zero and one part. 
If the variety $ \mathcal{V}_\mathbb{C}(q_1,\ldots,q_m )$ has dimension $n-m$, then $\mathcal{R}_F$ is invariant under the choice of generators of $\mathcal{I}(F)$.
\end{thm}
\begin{proof}
Consider the quadratic polynomials $G= ( g_1,\ldots,g_N )$ where $\mathcal{I}(G) = \mathcal{I}(F)$.
Then, for each $g \in G$ there exists some $a_1,\dots,a_m \in \mathbb{R}[x_1,\ldots,x_n]$  such that
\begin{equation}\label{eq:qa}
g = a_1 f_1 + \ldots + a_m f_m. 
\end{equation}
By \Cref{thm:sdp_under_linear}, it suffices to show that $a_i\in \RR$. 
We prove this by inducting on the minimum $d$ such that $a_1,\dots,a_m\in \RR[x_1,\dots,x_n]_{\leq d}$.

The base case $d = 0$ is immediate, so we assume $d >0$.
For $a_1,\dots,a_m$ in \eqref{eq:qa},~write  
\[ a_i = a_i^{(d)} + a_i'\]
where $a_i^{(d)}$ is homogeneous of degree $d$ or zero and $a'_1,\dots,a_m'\in \RR[x_1,\dots,x_n]_{<d}$. 
The degree $d+2$ part of $g$ in \eqref{eq:qa} is zero because $g$ is quadratic. This induces the relation
\[ 0= a_1^{(d)} q_1 + \ldots + a_m^{(d)} q_m. \]

Since $\dim(\mathcal{V}_{\mathbb{C}}(q_1,\ldots,q_m)) = n-m$, $(q_1,\ldots,q_m )$ form a regular sequence, meaning they only have Koszul syzygies. As a consequence, there exists a skew-symmetric matrix $M$ of homogeneous forms of degree at most $d-2$ such that
\[
\begin{pmatrix}
a_1^{(d)} \\ \vdots \\ a_m^{(d)}\end{pmatrix} = M \cdot \begin{pmatrix}
q_1 \\ \vdots \\ q_m
\end{pmatrix}.
\]

Substituting, we write
\[
\begin{pmatrix}
a_1-a_1' \\ \vdots \\ a_m - a_m'\end{pmatrix} = M \cdot \begin{pmatrix}
f_1 - p_1 \\ \vdots \\ f_m - p_m
\end{pmatrix}.
\]
Putting this together, we write
\[
\begin{pmatrix}
a_1 \\ \vdots \\ a_m
\end{pmatrix} = M \cdot \begin{pmatrix}
f_1 \\ \vdots \\ f_m 
\end{pmatrix} - M \cdot \begin{pmatrix}
p_1 \\ \vdots \\ p_m 
\end{pmatrix} + \begin{pmatrix}
a_1' \\ \vdots \\ a_m'
\end{pmatrix} = M \cdot \begin{pmatrix}
f_1 \\ \vdots \\ f_m 
\end{pmatrix}  + \begin{pmatrix}
a_1'' \\ \vdots \\ a_m''
\end{pmatrix}
\]
where $\deg(a_i'')\leq d-2 + 1 = d-1$. Since $M$ is skew-symmetric, and by taking the inner product of each side with $(f_1, \ \ldots, \ f_m)^T$, we see that 
\[ g = a_1'' f_1 + \ldots + a_m'' f_m. \]
We then use the induction hypothesis on $a_i''$ to conclude that 
\[g = \alpha_1 f_1 + \ldots + \alpha_m f_m \]
where $\alpha_i \in \mathbb{R}$. Therefore, each polynomial in $G$ can be written as an $\RR$-linear combination of polynomials in $F$, so by \Cref{thm:sdp_under_linear} we are done.
\end{proof}

Note that the variety $\mathcal{V}_{\CC}(q_1,\ldots,q_m)$ having dimension $n-m$ is equivalent to saying that $q_1,\ldots,q_m$ form a \emph{complete intersection}. \Cref{cor:sdp_welldefined} then can be strengthened to say that for almost all varieties defined by quadratic polynomials that have full monomial support, $\mathcal{R}_F$ is invariant under the choice of generators of $\mathcal{I}(F)$.

\begin{cor}
Consider $F = (  f_1,\ldots,f_m )$ where $\deg(f_i) = 2$ and each $f_i$ has full monomial support with general coefficients.
Then $\mathcal{R}_F$ is invariant under the choice of generators of $\mathcal{I}(F)$.
\end{cor}

Observe that in \Cref{ex:1} the SDP exact region, $\mathcal{R}_F$, is not invariant under choice of generators. We observe that the variety defined by the homogeneous degree two part of $F$ is $\{ x^2, xy \}$ and has dimension one, not zero. 
In addition,  for $F = ( f_1,\ldots, f_m )$ with $\deg(f_i) = 1$ for some $i$, the hypothesis of \Cref{cor:sdp_welldefined} is not satisfied. 
The next example shows how the SDP exact region need not be invariant under choice of generators in the latter case.

\begin{example}\label{ex:Rf_diff}
Consider $F = ( x-y)$ and $G = ( x-y, x^2 - xy + x - y )$. One can verify that $\mathcal{I}(F) = \mathcal{I}(G)$.  First observe that the Hessian of 
\eqref{QP} with feasible region defined by $F$, is $C$. Therefore any $(C,c) \in \mathcal{R}_F$ must have $C \succ 0$. We claim there exists $C \not\succ 0$ in $\mathcal{R}_{G}$. Consider $C = \begin{bmatrix}
-\frac{1}{2} & -1 \\ -1 & 5
\end{bmatrix}$ and $c = [2 \ \  2 ]^T$. Then taking $(\lambda_1, \lambda_2) = (- \frac{31}{5}, -1)$ and $x = -\frac{4}{5}$, one can check that $(C,c) \in \mathcal{R}_{G}$.
\hfill$\diamond$

\end{example}

\section{Symmetry in SDP exact regions $\mathcal{R}_F$} \label{sec:3}
In this section we study how the SDP exact region $\mathcal{R}_F$ is impacted by symmetries in $F$. 
Many applications naturally possess symmetry and exploiting this symmetry for computational gain has been explored in polynomial system solving and optimization \cite{wang2021tssos,amendola2021solving,lindberg2021exploiting,riener2013exploiting}. 

 For a polynomial system $F=0$, we encode its symmetry via finite subgroups of invertible $n \times n$ matrices, $GL_n(\mathbb{R})$. For a group $\mathfrak{G} \subset GL_n(\mathbb{R})$, we say $F$ is \emph{invariant} under the action of $\mathfrak{G}$ 
 if $F(g \cdot x) = 0$ for all $x \in \mathcal{V}_{\CC}(F)$ and $g \in \mathfrak{G}$. 
 In this case, we say $G$ \emph{acts} on the variety $\mathcal{V}_{\CC}(F)$ by standard matrix-vector multiplication $g \cdot x$. 
 For $x \in \mathcal{V}_{\CC}(F)$, the \emph{orbit} of $x$ is   
 \[
\mathcal{O}_x := \{y \in \mathcal{V}_{\mathbb{C}}(F) \ : \ y = g \cdot x \ \text{ for some } \ g \in \mathfrak{G} \}
 \]
 Observe that if $x \in \mathcal{O}_y$ then $y \in \mathcal{O}_x$, meaning that the orbits of $\mathcal{V}_{\CC}(F)$ partition $\mathcal{V}_{\CC}(F)$. For each orbit, we pick one representative and denote $\mathcal{V}_{\CC}(F)^\mathfrak{G} \subseteq \mathcal{V}_{\CC}(F)$ the set of these representatives. Specifically, $\mathcal{V}_{\CC}(F)^\mathfrak{G}$ has the property that for  
 distinct $x,y \in \mathcal{V}_{\CC}(F)^\mathfrak{G}$, 
 $\mathcal{O}_x \neq \mathcal{O}_y$ and $\mathcal{V}_{\CC}(F) = \bigcup_{x \in \mathcal{V}_{\CC}(F)^\mathfrak{G}} \{ g \cdot x \ : \ g \in \mathfrak{G} \}$. For clarity we demonstrate this with an example.

\begin{example}\label{ex:sym}
Consider the polynomial system $F(x) = x^T Ax - 1 = 0$ and the corresponding variety $\mathcal{V}_\mathbb{C}(F) = \{x \in \mathbb{C}^n \ : \ x^T A x =1 \}$ where $A \in \symnr$. The variety $\mathcal{V}_\mathbb{C}(F)$ has the symmetry that if $x \in \mathcal{V}_\mathbb{C}(F)$, then $-x \in \mathcal{V}_\mathbb{C}(F)$. We encode this symmetry with the group
\[ \mathfrak{G} = \left \{ \begin{pmatrix}
1 & & & \\ & 1 & & \\ & & \ddots & \\ & & & 1
\end{pmatrix}, \begin{pmatrix}
-1 & & & \\ & -1 & & \\ & & \ddots & \\ & & & -1
\end{pmatrix} \right \}  \]
where $\mathfrak{G}$ acts on $\mathcal{V}_\mathbb{C}(F)$ by matrix-vector multiplication. One can verify that for any $x \in \mathcal{V}_{\CC}(F), g \cdot x \in \mathcal{V}_{\CC}(F)$ for all $g \in \mathfrak{G}$. In this case $|\mathfrak{G}| = 2$, so $\mathfrak{G}$ partitions $\mathcal{V}_\mathbb{C}(F)$ into infinitely many orbits of size two, namely $\mathcal{O}_x = \{x, -x\}$.
\hfill$\diamond$
\end{example}

\Cref{ex:sym}
hints
that understanding $\mathcal{R}_F$ when $F$ is invariant under the action of a finite subgroup of $GL_n(\mathbb{R})$ is equivalent 
to understanding $\mathcal{R}_F$ when $F$ is the image of a linear transformation. 
Given $M\in GL_n(\mathbb{R})$, denote $\mathcal{R}_{M \cdot F}$ as the SDP exact 
region of $M \cdot F = \{f_1(Mx),\ldots,f_m(Mx) \}$ and denote the  variety  
\[ M \cdot \mathcal{V}_\mathbb{R}(F) := \{ M \cdot x \ : \ x \in \mathcal{V}_\mathbb{R}(F) \}. \]

\begin{thm}\label{thm:affine_sdp}
Suppose $(C,c) \in \mathcal{R}_F$ and consider the quadratic program
\begin{align*}
    \min_{x \in \mathbb{R}^n} \ f_0(x) \quad \text{subject to} \quad f_1(Mx) = \cdots = f_m(Mx) = 0 \tag{M-QCQP} \label{M-QP}
\end{align*}
 where $M \in GL_n(\mathbb{R})$. Then $(M^{-T}CM^{-1}, M^{-T}c) \in \mathcal{R}_{M\cdot F}$.
\end{thm}
\begin{proof}
Consider $M \in GL_n(\mathbb{R}) $ and recall that $M$ acts on $x \in \mathcal{V}_\mathbb{R}(F)$ by $x \mapsto Mx$. Using the change of coordinates induced by $M$, the quadratic program \eqref{eq:1} becomes:
\begin{align}
    \min_{x \in \mathbb{R}^n} \ g(M^{-1}x) \quad \text{subject to} \quad f_i(M^{-1}x) = 0, \ i \in [m]. \label{eq:M_inverse}
\end{align}
In matrix notation, $f_0(M^{-1}x)$ and $f_i(M^{-1}x)$ for $i \in [m]$ are:
\begin{align*}
    f_0(M^{-1}x) &= x^T M^{-T} CM^{-1} x + 2c^TM^{-1}x \\
    f_i(M^{-1}x) &= x^T M^{-T}A_i M^{-1} x + 2a_i^T M^{-1} x + \alpha_i.
\end{align*}
The Hessian of the Lagrangian of \eqref{eq:M_inverse} is
\[
\hat{H}(\lambda) = M^{-T}H(\lambda) M^{-1}
\]
where $H(\lambda)$ is the Hessian of the Lagrangian associated to \eqref{eq:1}. It is clear that if $H(\lambda) \succ 0$ for some $\lambda \in \mathbb{R}^m$ then $\hat{H}(\lambda) \succ 0$.

Writing out the definition of the SDP exact region for \eqref{eq:M_inverse} we have
\begin{align*}
    M^{-T} c - \sum_{i=1}^m \lambda_i M^{-T} a_i + M^{-T} H(\lambda) M \hat{x} &= 0 \qquad \text{for } \hat{x} \in \mathcal{R}_{MF} 
\end{align*}
Since $\hat{x} = Mx$ for some $x \in \mathcal{V}_\mathbb{R}(F)$ we have 

\begin{align*}
    M^{-T}(c - \sum_{i=1}^m \lambda_i a_i + H(\lambda)x) = 0
\end{align*}
Since $M^{-T}$ is full rank, if $(C,c) \in \mathcal{R}_F$, then $(M^{-T}CM^{-1}, M^{-T}c) \in \mathcal{R}_{M \cdot F}$.
\end{proof}

We wish to translate \Cref{thm:affine_sdp} into a statement about the structure of $\mathcal{R}_F$. To do this, we remark that the authors in \cite{cifuentes2020the} observe that $\mathcal{R}_F$ can always be expressed as the union of spectrahedral shadows, namely one for each $x \in \mathcal{V}_{\RR}(F)$. To see this, observe that for fixed $x \in \mathcal{V}_\mathbb{R}(F)$, \Cref{def:sdpExact} defines a spectrahedron in the space $(\lambda, C, c)$. By projecting this spectrahedron onto the space $(C,c)$, for each $x \in \mathcal{V}_\mathbb{R}(F)$ we then get a spectrahedral shadow, $S_x$. We can then write $\mathcal{R}_F = \bigcup_{x \in \mathcal{V}_\mathbb{R}(F)} S_x $. 

Suppose a finite group $\mathfrak{G} \subset GL_n(\RR)$ acts on $F = (f_1,\ldots,f_m )$. The above observation and \Cref{thm:affine_sdp} then describe a way to partition these spectrahedral shadows via the action of $\mathfrak{G}$.

\begin{cor}
Consider a group $\mathfrak{G} \subseteq GL_n(\mathbb{R}) $ that acts on $F$. The spectrahedral shadows, $S_x$ and $S_{g\cdot x}$, defined as the SDP exact regions corresponding to $x$ and $g \cdot x$ are in bijection via the mapping $(C,c) \to (g^{-T}Cg^{-1}, g^{-T}c)$. Moreover, $\mathfrak{G}$ partitions $\mathcal{R}_F$ as
\[\mathcal{R}_F = \bigcup_{x \in \mathcal{V}_{\RR}(F)^\mathfrak{G}} \{ S_{g \cdot x} \ : \ g \in \mathfrak{G}\}. \]
\end{cor}

\begin{example}\label{ex:2}
Consider the optimization problem 
\begin{align*}
\min_{x_1,x_2 \in \mathbb{R}} \ x_2^2 + \cij_{12}x_1x_2 + 2\ck_1x_1 \quad \text{subject to} \quad x_1^2 =1.
\end{align*}
From \Cref{ex:sym}, we know the group $\mathfrak{G} = \left \{ \begin{pmatrix}
1 & 0 \\ 0 & 1 
\end{pmatrix}, \begin{pmatrix}
-1 & 0 \\ 0 & -1 
\end{pmatrix} \right \}$ acts on $F$ and partitions $\mathcal{V}_\mathbb{R}(F)$ into orbits of size two. This symmetry is seen in \Cref{fig:semialgSets} as $\mathcal{R}_F$ is symmetric around the $\ck_1$ and $\cij_{12}$ axes.
\hfill$\diamond$
\end{example}

\begin{figure}[h!]
    \centering
    \includegraphics[width = 0.4\textwidth]{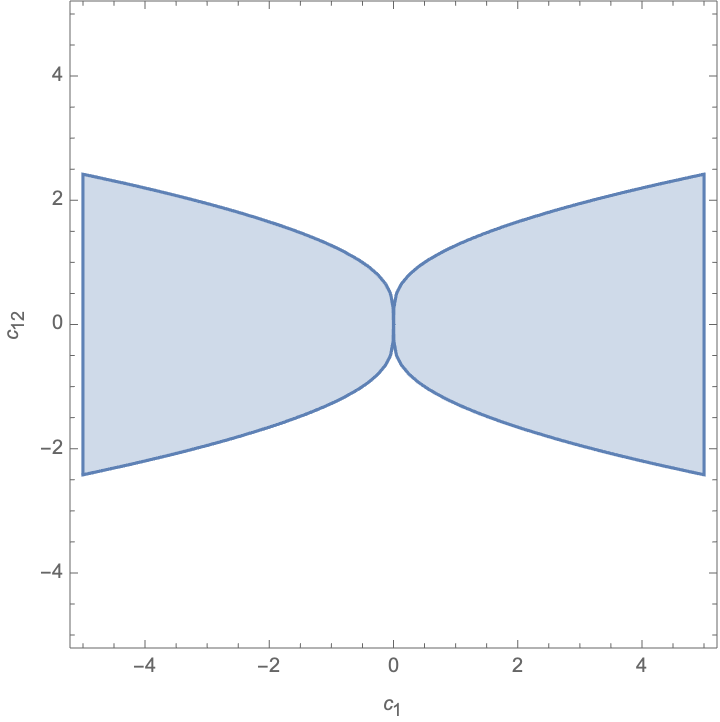}
    \caption{$\mathcal{R}_F$ from the optimization problem defined in \Cref{ex:2}.}
    \label{fig:semialgSets}
\end{figure}

\section{Binary quadratic programs}\label{sec:4}
In this final section we leverage our new understanding of invariants of SDP exact regions to solve quadratic binary programs. As a by-product, we develop a heuristic for tackling an NP hard problem with computational results to motivate future research.

\subsection{SDP exact regions of binary quadratic programs}
We consider the family of quadratic programs with feasible set $\pmOne^n$. These are programs of the~form:
\begin{align}
    \min_{x \in \mathbb{R}^n} \ x^T C x +  2c^T x \quad \text{subject to} \quad f_i(x):= x_i^2 -1 = 0, \ i \in [n]. \tag{BQP} \label{QBP}
\end{align}
Binary quadratic programs model a wide class of combinatorial optimization problems and in general are NP hard. Multiple solution methods have been proposed for this class of problems, including the famous SDP approximation algorithm by Goemans and Williamson which gives a polynomial time algorithm to the max-cut problem that returns a solution with an optimality gap of at least $0.878$  \cite{goemans1995improved}. For a survey on relevant approaches and applications, see \cite{kochenberger2014the}.

In this section we wish to understand the set of $(C,c)$ for which the Shor relaxation of \eqref{QBP}  is exact. First, by \Cref{cor:sdp_welldefined} the SDP exact region of \eqref{QBP} is invariant under any choice of generators that have $\pmOne^n$ as its complex variety. Next, using \Cref{thm:affine_sdp} we are able to succinctly classify the SDP exact region of \eqref{QBP}.

\begin{thm}\label{thm:qbp_exact}
The SDP exact region of \eqref{QBP} consists of $2^n$ spectrahedra, which are invertible linear transformations of 

\[
\mathcal{S} = \{(C,c) \ : \ \begin{bmatrix}
- (c_1 + \sum_{i=2}^n c_{1i}) & c_{12} & \cdots & c_{1n} \\
c_{12} & - (c_2 + c_{12} + \sum_{i=3}^n c_{2i}) & \cdots & c_{2n} \\
\vdots & & \ddots & \vdots \\
c_{1n} & &  & - (c_n + \sum_{i=1}^{n-1} c_{in} )
\end{bmatrix} \succ 0 \}.
\]
Specifically, using variable ordering $(c_{1,2},\ldots, c_{n-1,n}, c_1,\ldots,c_n)$, we can write $\mathcal{R}_F = \{g \cdot \mathcal{S} \ : \ g \in \mathfrak{G}\}$ where $\mathfrak{G} \subset GL_{\frac{n(n+1)}{2}}(\mathbb{R})$ is the matrix group 
\[
\mathfrak{G} = \{ \Diag(x_1 x_2 , \ldots, x_1x_n,x_2x_3,\ldots,x_2x_n,\ldots , x_{n-1}x_n, x_1,\ldots, x_n ) : x \in \pmOne^n \}.
\]
Moreover, $\mathfrak{G}$ is isomorphic to  $\underbrace{\mathbb{Z}_2 \times \cdots \times \mathbb{Z}_2}_n$.
\end{thm}
\begin{proof}
By direct computation we see that $\mathcal{R}_F$ consists of the union of $2^n$ spectrahedra, each one given by 
\begin{align}
H(x) = \begin{bmatrix}
- \frac{1}{x_1}(c_1 + \sum_{i=2}^n c_{1i}x_i) & c_{12} & \cdots & c_{1n} \\
c_{12} & - \frac{1}{x_2}(c_2 + c_{12}x_1 + \sum_{i=3}^n c_{2i}x_i) & \cdots & c_{2n} \\
\vdots & & \ddots & \vdots \\
c_{1n} & &  & - \frac{1}{x_n}(c_n + \sum_{i=1}^{n-1} c_{in} x_i)
\end{bmatrix} \succ 0 \label{H_QBP}
\end{align}
for $x \in \pmOne^n$. 
Using \Cref{thm:affine_sdp} we see each spectrahedron is linearly equivalent to $\mathcal{S}$ where the linear map is defined is defined by an element of $\mathfrak{G}$. It is clear that $\mathfrak{G}$ consists of $2^n$ elements, each with order two, therefore, $\mathfrak{G} \cong \mathbb{Z}_2 \times \cdots \times \mathbb{Z}_2$.
\end{proof}

\begin{example}
Consider when $n = 2$. 
The SDP exact region of \eqref{QBP} for $n=2$ consists of $4$ spectrahedra in $\sym_2 (\mathbb{R}) \times \RR^2 \cong \RR^5$. Observe that the diagonal entries of $C \in \sym_2(\RR)$ do not affect SDP-exactness, therefore we consider $\mathcal{R}_F$ in the three dimensional subspace whose coordinates are $(c_{12},c_1,c_2)$, 
as shown in \Cref{fig:sdp_exact_n_2}. These spectrahedra are 
$H(x)\succ 0$ for 
$x \in \{ (-1,-1), (1,1), (-1,1),(1,-1) \}$. Using the notation in \eqref{H_QBP},

\begin{figure}[h!]
    \centering
    \includegraphics[width = 0.4 \textwidth]{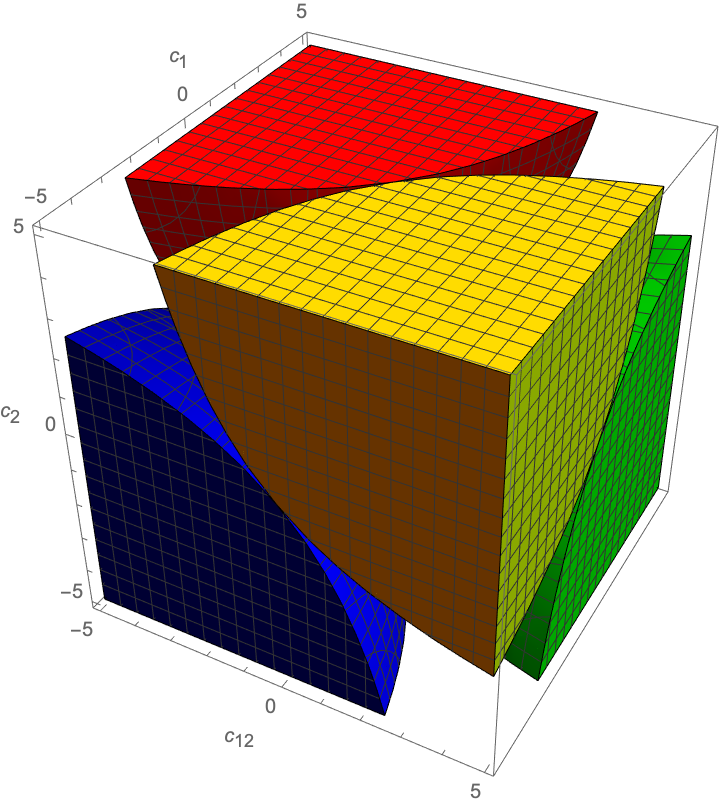}
    \caption{$\mathcal{R}_F$ for \eqref{QBP} when $n=2$.}
    \label{fig:sdp_exact_n_2}
\end{figure}

\begin{align*}
    H(1,1) &= \begin{bmatrix} -c_{12} - c_1& c_{12} \\
    c_{12}& -c_{12} - c_2 \end{bmatrix} \qquad &&H(-1,1) = \begin{bmatrix} c_{12} + c_1 &  c_{12} \\ c_{12} &  c_{12} - c_2 \end{bmatrix} \\
    H(-1,-1) &= \begin{bmatrix} -c_{12} + c_1&  c_{12} \\ c_{12} & -c_{12} + c_2 \end{bmatrix} \qquad &&H(1,-1) = \begin{bmatrix} c_{12} - c_1 & c_{12} \\ c_{12} & c_{12} + c_2 \end{bmatrix},
\end{align*}
and we
call the spectrahedra corresponding to $H(x) \succ 0$,  $S_{1,1},S_{-1,1},S_{-1,-1},S_{1,-1}$. 
By \Cref{thm:qbp_exact}
every spectrahedron is linearly equivalent to $S_{1,1}$. 
Observe
\begin{align*}
    S_{-1,-1} &= \begin{bmatrix} 1 & 0 & 0 \\ 0 & -1 & 0 \\ 0 & 0 & -1 \end{bmatrix} \cdot  S_{1,1}, \qquad 
    S_{-1,1} = \begin{bmatrix} -1 & 0 & 0 \\ 0 & -1 & 0 \\0 & 0 & 1 \end{bmatrix} \cdot S_{1,1}, \qquad
    S_{1,-1} = \begin{bmatrix} -1 & 0 & 0 \\ 0 & 1 & 0 \\ 0 & 0 & -1 \end{bmatrix} \cdot S_{1,1}.
\end{align*}
We see that the set of matrices transforming $S_{1,1}$ to the other spectrahedron form a group. Note that $\begin{bmatrix} c_{11} & c_{12} \\ c_{12} & c_{22} \end{bmatrix} \succ 0$ and $\begin{bmatrix} c_{11} & -c_{12} \\ -c_{12} & c_{22} \end{bmatrix} \succ 0$ define the interior of the same spectrahedron. In this case $\mathcal{R}_F = \{x \cdot S_{1,1} \ : \ x \in \mathfrak{G} \}$ where
\[
\mathfrak{G} = \left \{\begin{bmatrix} 1 & 0 & 0 \\ 0 & 1 & 0 \\ 0 & 0 & 1 \end{bmatrix}, \begin{bmatrix} 1 & 0 & 0 \\ 0 & -1 & 0 \\ 0 & 0 & -1 \end{bmatrix}, \begin{bmatrix} -1 & 0 & 0 \\ 0 & -1 & 0 \\0 & 0 & 1 \end{bmatrix}, \begin{bmatrix} -1 & 0 & 0 \\ 0 & 1 & 0 \\ 0 & 0 & -1 \end{bmatrix} \right \} \cong \mathbb{Z}_2 \times \mathbb{Z}_2.
\]

A picture of $\mathcal{R}_F$ is shown in \Cref{fig:sdp_exact_n_2} where the red region is $S_{1,1}$, the green region is $S_{-1,1}$, the yellow region is $S_{1,-1}$ and the blue region is $S_{-1,-1}$.  
\hfill$\diamond$

\end{example}

\subsection{Finding candidate solutions to \eqref{QBP}}
Recall from \Cref{thm:qbp_exact} that determining if \eqref{QBP} has an exact SDP relaxation is equivalent to determining if $H(x) \succ 0$ for some $x \in \pmOne^n$ where $H(x)$ is as defined in \eqref{H_QBP}. Note that in general, testing whether a given objective function leads to an exact relaxation of \eqref{QP} for binary optimization problems is NP hard \cite{laurent1995on}. 

A necessary condition for $H(x) \succ 0 $ is for the diagonal entries of $H(x)$ to be positive. Observe that these entries are linear, so for each $x \in \pmOne^n$ we have a corresponding polyhedral cone, $P_x$, such that $P_x\supset S_x$. Specifically for each $x \in \pmOne^n$, we define 
\begin{align*}
P_x &= \{(C,c) \in \sym_n(\RR) \times \RR^n \ : \ -\frac{1}{x_j}(c_j + \sum_{\substack{i=1\\i \neq j}}^n c_{ij}x_i)>0 \ \text{ for all } j \in [n] \}  \\
&= \{(C,c) \in \sym_n(\RR) \times \RR^n \ : \ - \frac{1}{x_1}(c_1 + \sum_{i=2}^n c_{1i}x_i) >0, \ \ldots, - \frac{1}{x_n} (c_n + \sum_{i=1}^{n-1} c_{in} x_i)>0 \}
\end{align*}
Checking if $x \in P_x$ amounts to checking if $n$ linear equations are positive. By taking the sign of each entry, for each $P_x$ we can then associate an $n$-tuple, $v \in \pmOne^n$.

Aside from using $P_x$ to determine if a particular instance of \eqref{QBP} has an exact SDP relaxation, understanding for which $x \in \pmOne^n$, $(C,\ck) \in P_x$ is important in terms of globally solving \eqref{QBP}. This is explained in the following proposition.

\begin{prop}\label{prop:Px_global}
Let $x^* \in \pmOne^n$ be a unique global minimizer to \eqref{QBP}. Then $(C,\ck) \in P_{x^*}$.
\end{prop}
\begin{proof}
If $x^*$ is a unique global minimizer to \eqref{QBP}, then $(x^*)^TCx^* + 2c^T x^* < x^T C x + 2c^T x$ for all $x \in \pmOne^n \backslash x^*$.
Recall that
\[P_{x^*} = \left \{(C,c) \ : \  - \frac{1}{x^*_1}(c_1 + \sum_{i=2}^n c_{1i}x^*_i) >0 , \ldots  , - \frac{1}{x^*_n}(c_n + \sum_{i=1}^{n-1} c_{in} x_i^*)>0 \right \} \]
Specifically, the $k$th inequality in $P_{x^*}$ is of the form:
\[
- \frac{1}{x_k^*}(c_k+ \sum_{\substack{i=1 \\i \neq k}}c_{ik}x_i^*) >0.
\]
Consider $y \in \pmOne^n$ that differs from $x^*$ in only the $k$th component. Then 
\begin{align*}
   & (x^*)^T C x^* + 2 c^T x^* < y^T C y + 2c^T y \\  &\Rightarrow \sum_{i=1}^n x_i^* c_i + \sum_{i=1}^n \sum_{j=i+1}^n c_{ij}x_i^* x_j^* < \sum_{i=1}^n y_i c_i + \sum_{i=1}^n \sum_{j=i+1}^n c_{ij}y_i y_j \\
    &\Rightarrow \sum_{i=1}^n x_i^* c_i + \sum_{i=1}^n \sum_{j=i+1}^n c_{ij}x_i^* x_j^* <\sum_{\substack{i=1 \\ i\neq k}}^n x_i^* c_i - x_k^* c_k + \sum_{\substack{i=1 \\ i\neq k}}^n \sum_{j=i+1}^n c_{ij}x_i^* x_j^* - \sum_{\substack{i=1 \\ i \neq k}}^n x_i^*x_k^* c_{ik} \\
    &\Rightarrow  - x_k^* c_k - \sum_{\substack{i=1 \\ i \neq k}}^n x_i^* x_k^* c_{ik} > 0 \\
    &\Rightarrow -x_k^* (c_k + \sum_{\substack{i=1 \\ i \neq k}}^n c_{ik}x_i^*) >0.
\end{align*}
Since $x^* \in \pmOne^n$, this is equivalent to the $k$th inequality of $P_{x^*}$.
\end{proof}

\begin{rem}
Observe that in \Cref{prop:Px_global} we assume $x^*$ is a \emph{unique} global minimizer. We see this assumption is necessary since our definition of $P_x$ is that of an open cone. Throughout this section we consider objective functions for which $x^T C x + 2 c^T x \neq y^T C y + 2 c^T y$ for any distinct $x,y \in \pmOne^n$. We refer to such objective functions as \emph{generic} since the set of objective functions we disregard lie on a proper Zariski closed set. In other words, the objective functions we exclude lie on a set of measure zero in $\symnr \times \RR^n$.
\end{rem}

\begin{algorithm}[h!]
\KwIn{
$(C,c) \in \symnr \times \mathbb{R}^n$
}
\KwOut{
   $x \in \pmOne^n$ such that $(C,c) \in P_x$
    }
\begin{enumerate}
    \item Fix $x \in \pmOne^n$.
    \item\label{alg:stepP} If $(C,c) \in P_{x}$, terminate and return $x$.
    \item If $(C,c) \not\in P_{x}$ then let $i^*$ be the first index of $v \in \mathbb{R}^n$ such that $v_{i^*}<0$. Let $\hat{x} \in \pmOne^n$ be the vector such that 
    \[
    \hat{x}_j = \begin{cases}
    x_j & \text{ for } j \neq i^* \\
    - x_j & \text{ for } j = i^*.
    \end{cases}
    \] \label{alg:step2}
    \item Replace $x$ with $\hat x$ and return to Step~\ref{alg:stepP}.
\end{enumerate}
\caption{Find $x\in \pmOne^n$ where $(C,c) \in P_x$}
\label{alg:Px}
\end{algorithm}

\Cref{prop:Px_global} shows the importance of understanding all $x \in \pmOne^n$ where $(C,\ck) \in P_x$.
\Cref{alg:Px} gives a method to find one $x \in \pmOne^n$ where $(C,c) \in P_x$. Every time Step~\ref{alg:step2} is repeated, we jump from $x \in \pmOne^n$ to $\hat{x} \in \pmOne^n$ where $\hat{x}$ only differs from $x$ in the $i$th coordinate. 

To follow the path this algorithm takes, we consider the bipartite graph, $G$, consisting of nodes $V_E, V_O$. Each $x \in \pmOne^n$ that differs from $(1,\ldots,1)$ in an even number of places defines a node in $V_E$ and each $x \in \pmOne^n$ that differs from $(1,\ldots,1)$ in an odd number of places defines a node in $V_O$. There is an edge between $x \in V_E$ and $y \in V_O$ if $x$ and $y$ differ in exactly one place. 

Each edge has a linear expression in the entries of $(C,c)$ associated to it, given by one of the inequalities of $P_x$. Namely, consider an edge $e = xy$ where $x \in V_E$ and $y \in V_O$. Suppose that $x$ and $y$ differ in the $i$th place. Then, $e$ has the $i$th inequality of $P_x$ associated to it. If this inequality is positive we direct $e$ to be pointing to the left i.e. $e = yx$. If it is negative we direct $e$ to be pointing to the right i.e. $e = xy$. For a fixed $(C,c)$, $G$ is a directed bipartite graph with $2^n$ vertices and $n \cdot 2^{n-1}$ edges.

We write the sequence \Cref{alg:Px} takes as $\{e_{i_1},\ldots,e_{i_k} \}$ where $e_{i_j}$ reflects the fact that in Step~\ref{alg:step2} we negate the $i_j$th entry of $x$. We express the vertices that \Cref{alg:Px} transverses starting from $x$ as
\[\left \{ x, x(e_{i_1}),x(e_{i_1}(e_{i_2})), \ldots, x(e_{i_1}(e_{i_2}(\cdots(e_{i_k})\cdots))\right\} .\] 
Each time we repeat Step~\ref{alg:step2}, another inequality is induced on the entries of $(C,c)$. We give an example of $G$ when $n = 2$.

\begin{example}
Consider when $n = 2$. The graph on the left shows the general bipartite structure for $n = 2$ before the objective function is fixed. 
The graph on the right show the directed version of $G$ for objective function given by $c_1 = 5, c_2 = -1, c_{12} = 3$. In this case, we see that $G$ has a unique sink at $(-1,1)$. By \Cref{prop:Px_global}, this tells us that $(-1,1)$ is the unique global optimal solution.

\begin{tikzpicture}
      \tikzset{enclosed/.style={draw, circle, inner sep=0pt, minimum size=.15cm, fill=black}}
      \tikzset{
    edge/.style={->,> = latex'}
}

      \node[enclosed, label={left, yshift=.2cm: $(1,1)$}] (11) at (0.75,3.25) {};
      \node[enclosed, label={right, yshift=0cm: $(-1,1)$}] (-11) at (4.5,3.25) {};
      \node[enclosed, label={below, xshift=.2cm: $(1,-1)$}] (1-1) at (4.5,0.75) {};
      \node[enclosed, label={left, yshift=-.2cm: $(-1,-1)$}] (-1-1) at (0.75,0.75) {};

      \draw (-1-1) -- (1-1) node[midway,  below] (edge4) {$c_1 - c_{12}$};
      \draw (11) -- (-11) node[midway, above] (edge6) {$-c_1 - c_{12}$};
      \draw (11) -- (1-1) node[near start,  left] (edge7) {$-c_2 - c_{12}$};
      \draw (-1-1) -- (-11) node[near end, right] (edge9) {$c_2 - c_{12}$};
      
    \node[enclosed, label={left, yshift=.2cm: $(1,1)$}] (112) at (8.75,3.25) {};
      \node[enclosed, label={right, yshift=0cm: $(-1,1)$}] (-112) at (12,3.25) {};
      \node[enclosed, label={below, xshift=.2cm: $(1,-1)$}] (1-12) at (12,0.75) {};
      \node[enclosed, label={left, yshift=-.2cm: $(-1,-1)$}] (-1-12) at (8.75,0.75) {};

      \draw[edge]  (1-12) -- (-1-12) node[below, midway] (edge5) {$2$};
      \draw[edge] (112) -- (-112) node[midway, above] (edge6) {$-8$};
      \draw[edge] (112) -- (1-12) node[near start,  left] (edge7) {$-2 \ $};
      \draw[edge] (-1-12) -- (-112) node[near end, right] (edge9) {$-4$};

\end{tikzpicture}

\hfill$\diamond$
\end{example}

For \Cref{alg:Px} to terminate, $G$ needs to be acyclic. Before we show this, we need a lemma that helps to keep track of the inequalities induced by Step~\ref{alg:step2} of \Cref{alg:Px}. For ease of notation and without loss of generality, we now assume we initiate \Cref{alg:Px} at $x = (1,\ldots,1)$.

\begin{lem}\label{lem:ineq}
Consider a sequence $e_{i_1},\ldots, e_{i_k}$ starting from $x = (1,\ldots,1)$. Let $s_i$ be the number of times $e_i$ appears in this sequence.  Suppose $e_j$ appears next in the sequence, then this induces the inequality
\begin{align}\label{alg:inequality}
(-1)^{s_j} c_j +  \sum_{\substack{i=1 \\ i \neq j}}^n (-1)^{s_i + s_j} c_{ij} > 0.
\end{align}
\end{lem}
\begin{proof}
This is by definition of $P_x$.
\end{proof}

\begin{thm}
For generic $(C,c) \in \symnr \times \RR^n$, \Cref{alg:Px} terminates. Moreover, for generic $(C,c)$ there exists at least one $x \in \pmOne^n$ such that $(C,c) \in P_x$. 
\end{thm}
\begin{proof}
It suffices to show that $G$ is acylic. For the sake of contradiction, suppose there is a cycle of length $2k$ in $G$, $\mathcal{C}= \{x_1,\ldots,x_{2k} \}$. Each vertex in the cycle implies an inequality of the form \eqref{alg:inequality} on the entries of $(C,c)$. We claim the sum of these inequalities is $0$. 

Let $S = \{i_1,\ldots,i_{2k}\}$ be the set of moves used to form $\mathcal{C}$. Specifically, for $j \in [2k]$, $x_j = x_1(e_{i_1}(\cdots(e_{i_{j}}) \cdots )$. Since $\mathcal{C}$ is a cycle, every $i_j \in S$ appears an even number of times. This tells us that $c_j$ appears in an even number of inequalities. By \Cref{lem:ineq}, when it first appears it is negative, when it next appears it is positive and so on. This shows that summing these inequalities results in $c_j$ being canceled out for any $j \in S$.

Now we claim that $\cij_{ij}$ also appears in an even number of inequalities. If $i \not\in S$ and $j \not\in S$ then $\cij_{ij}$ does not appear in a single inequality. Otherwise, $\cij_{ij}$ is in an inequality every time $i$ and $j$ appear in $S$. Each time $\cij_{ij}$ appears it alternates sign, and since it appears an even number of times, summing all inequalities results in $\cij_{ij}$ being canceled out.
Therefore summing all inequalities gives $0<0$ which is a contradiction, so no cycle can exist.

Since our graph is a directed acyclic graph, there exists a topological sorting of the vertices. This means there exists a vertex with all arrows pointing in. The vertex $x$ with all arrows pointing in then satisfies $(C,c) \in P_x$. In addition, as we follow a path from some starting vertex we will not visit any vertex twice. Since there are finitely many vertices and at every iteration we move to a new, distinct vertex, this means eventually we will end up at a vertex $x$ such that $(C,c) \in P_x$.

\end{proof}

We present empirical results on how many iterations it takes to run \Cref{alg:Px} in \Cref{tab:Px_stats}. We select each $(C,c)$ randomly by sampling each element $\cij_{ij}, \ck_i$, $i,j\in [n]$,  to be 
independent and identically distributed from the standard normal distribution.
We see on average that \Cref{alg:Px} terminates in less than $n$ iterations and it always terminates in less than $2n$ iterations. This motivates \Cref{conj:1}.

\begin{table}[h!]
    \centering
    \begin{tabular}{|c|c|c|c|c|c|c|c|c|c|c|}
    \hline 
        $n$ & $10$ & $20$ & $30$ & $40$ & $50$ & $60$ & $70$ & $80$ & $90$ & $100$\\
        \hline 
        Maximum & $16$ & $29$ & $44$ & $63$ & $87$ & $99$ & $115$ & $126$ & $146$ & $171$\\
        Minimum & $1$ & $3$ & $7$ & $8$ & $15$ & $24$ & $31$ & $37$ & $47$ & $58$\\
        Mean & $6.21$ & $13.39$ & $21.81$ & $31.82$ & $42.19$ & $53.46$ & $66.06$ & $78.34$ & $91.30$ & $104.81$\\
        Median & $6$ & $13$ & $21$ & $31.5$ & $42$ & $53$ & $65$ & $77$ & $90$ & $105$\\
        $\sigma^2$ & $2.49$ & $4.42$ & $6.39$ & $8.25$ & $9.99$ & $11.94$ & $14.09$ & $15.54$ & $17.54$ & $19.22$\\
        \hline 
    \end{tabular}
    \caption{Statistics on the number of iterations it took for \Cref{alg:Px} to terminate using random $(C,c) \in \symnr \times \mathbb{R}^n$ in a trial of $1,000$.}
    \label{tab:Px_stats}
\end{table}

\begin{conj}\label{conj:1}
\Cref{alg:Px} terminates after repeating Step~\ref{alg:step2} at most $2n$ times.
\end{conj}

In addition, we observe in our computations that there is a limit to how many $P_x$ can intersect nontrivially. We first show when $P_x \cap P_{x'} \neq \emptyset$.

\begin{prop}\label{prop:nontrivial_Px_intersect}
For fixed $x, x' \in \pmOne^n$, $P_x \cap P_{x'} \neq \emptyset$ if and only if $x'$ differs from $x$ in more than one place.
\end{prop}
\begin{proof}
$(\Rightarrow) : $ Assume $P_x \cap P_{x'} \neq \emptyset$ for some $x,x' \in \pmOne^n$. If $x,x'$ only differed the $i$th component, the $i$th inequality of $P_{x}$ is of the form $\ell >0$ and the $i$th inequality of $P_{x'}$ is of the form $- \ell >0$. Clearly both can not be true, so $x$ and $x'$ must differ in at least two places. \\
$(\Leftarrow): $ Assume that $x$ and $x'$ differ in at least two places, we want to show there exists some $(C,c) \in P_x \cap P_{x'}$. Let $P_x = \{(C,c) \ : \ \ell_1(x) >0, \ldots, \ell_n(x) >0 \}$ and $P_{x'} = \{(C,c) \ : \ \ell_1(x') >0, \ldots, \ell_n(x')>0 \}$ where
\begin{align*}
 \ell_i(x) = -\frac{1}{x_i} (\ck_i + \sum_{\substack{j = 1\\ j \neq i}}^n \cij_{ij}x_j).
\end{align*}
Observe that $\ell_i(x) \neq - \ell_i(x')$ for any $i \in [n]$ since $x$ differs from $x'$ in more than one place. Also, observe that for $\ell_i(x)$ and $\ell_j(x)$ the coefficient in front of $\cij_{ij}$ is $-\frac{x_j}{x_i}$ and $- \frac{x_i}{x_j}$ respectively. Therefore, for any $x \in \pmOne^n$, the coefficient of $\cij_{ij}$ is the same in both $\ell_i(x)$ and $\ell_j(x)$. Since $\ell_i(x) \neq - \ell_i(x')$ for any $i \in [n]$, then for each $i$ there exists at least one $ s \in \{\ck_i, \cij_{i1},\ldots, \cij_{in} \}$ such that the coefficient of $s$ in both $\ell_i(x)$ and $\ell_i(x')$ is the same. Moreover, if $s = \cij_{ij}$ for some $j \in [n]$, then the coefficient of $\cij_{ij}$ in $\ell_j(x), \ell_j(x'), \ell_i(x)$ and $\ell_i(x')$ is the same.

With this in mind, consider $\ck_i, \cij_{ij} = 0$ if the coefficients of $\ck_i$ and $\cij_{ij}$ have different signs in $\ell_i(x)$ and $\ell_i(x')$ and take $\ck_i = \sign(- \frac{1}{x_i})$ and $\cij_{ij} = \sign(- \frac{x_j}{x_i}) $ otherwise. Then $\ell_i(x) >0$ and $\ell_i(x')>0$ for all $i \in [n]$, giving  $(C,\ck) \in P_x \cap P_{x'}$.

\end{proof}

\Cref{prop:nontrivial_Px_intersect} shows that if $(C,\ck) \in P_x$ for some $x \in \pmOne^n$, then $(C,c)$ could potentially be in $P_{x'}$ for $2^n - n - 1$ other $x' \in \pmOne^n$. While this seems then like knowing $(C,\ck) \in P_x$ doesn't give much information, in all computations we observed that $(C,\ck) \in P_x$ for at most $n$ distinct $x \in \pmOne^n$.

\begin{conj}
$(C,\ck) \in P_x$ for at most $n$ distinct $x \in \pmOne^n$.
\end{conj}

\subsection{Comparing Algorithms}
We conclude by comparing \Cref{alg:Px} to the standard Shor relaxation for \eqref{QBP}. The intuition is that choosing $x$ such that $(C,c) \in P_x$ gives a good candidate for an optimal solution to \eqref{QBP}. First, note that any optimal solution to \eqref{QP-Relax} gives a lower bound on the optimal value of \eqref{QBP}. In contrast, any $x \in \pmOne^n$ gives an upper bound on the optimal value. Therefore, when comparing the two methods, it makes most sense to compare the range that these optimal solutions take. 

For fixed $(C,c) \in \symnr \times \mathbb{R}^n$, we let 
\[
p_{x} := x^T C x + 2c^T x
\] where $x$ is the output of
\Cref{alg:Px}. Let $p_{SDP}$ be the optimal value of the Shor relaxation of \eqref{QBP}. 
To compare $p_x$ and $p_{SDP}$, define 
$\Dpp:\mathbb{R}\times \mathbb{R}\to [0,\frac{1}{2}]$, 
\[ \Dpp(p_{x}, p_{SDP}) =  \frac{p_{x} - p_{SDP}}{\frac{1}{2}(|p_{x}| + |p_{SDP}|)}. \]
Note that 
$\Dpp(p_{x}, p_{SDP}) = 0$ if and only if $p_{x} = p_{SDP}$, so  only when they both give optimal solutions to \eqref{QBP}. 
On the other side, $\Dpp(p_{x}, p_{SDP}) = \frac{1}{2}$ if and only if $p_{x} = 0$ or $p_{SDP} = 0$. 
The numerator of $\Dpp(p_{x}, p_{SDP})$ gives the range the optimal values take whereas the denominator normalizes this range based on the magnitude of these values. 
Therefore, the closer $\Dpp(p_{x}, p_{SDP})$ is to zero, 
the tighter the interval $[p_x, p_{SDP}]$ is around true optimal value of \eqref{QBP}.

In addition, we can consider initializing \Cref{alg:Px} with $M$ random binary vectors instead of just one. This gives up to $M$ distinct $x$ such that $(C,c) \in P_x$. Denote this collection of $x$ as $V$. Then setting, $p^{(M)}_{x} = \min \{ x^T C x + 2c^T x \ : x \in V \}$ we see that {so long as $x \in V$,} $p^{(M)}_{x} \leq p_x$.

\begin{table}[h!]
    \centering
    \begin{tabular}{|c|c|c|c|c|c|c|}
    \hline 
        $n$ & $50$ & $100$ & $150$ & $200$ & $250$ & $300$ \\
        \hline 
        Time $p_x$ (sec) & $0.004$ & $0.028$ & $0.121$ & $0.240$ & $0.535$ & 0.913 \\
         Time  $p^{(M)}_{x}$ (sec)  & $0.014$ & $0.134$ & $0.664$ & $1.697$ & $3.790$ & $6.902$ \\
         Time  $p_{SDP}$ (sec)  & $0.396$ & $2.011$ & $5.621$ & $12.022$ & $21.122$ & $35.299$ \\
        $\Dpp(p_{x}, p_{SDP})$& $0.257$ & $0.267$ & $0.269$ & $0.274$ & $0.277$ & $0.280$\\
       $\Dpp(p^{(M)}_{x}, p_{SDP})$ & $0.186$ & $0.210$ & $0.220$ & $0.230$ & $0.238$ & $0.243$\\
        \hline 
    \end{tabular}
    \caption{Comparison of the time it took to find $p_x, p^{(M)}_{x}$ and  $p_{SDP}$ for $M = \lceil \log (n) \rceil $, and the range these values took in a trial of $1,000$ randomly generated $(C,c)$.}
    \label{tab:Px_sdp_stats}
\end{table}

We compare the time it took to find $p_x, p_x^{(M)}$ and $p_{SDP}$ and their values by considering $1000$ random instances of \eqref{QBP}. We run our simulations using \texttt{Julia 1.7} on a $2018$ Macbook Pro with 2.3 GHz Quad-Core Intel Core i5 processor. We use the package \texttt{SCS.jl} to solve all semidefinite programs. As before, we sample each $c_i,c_{ij}$ independently from a $\mathcal{N}(0,1)$ distribution. The results of these simulations are given in \Cref{tab:Px_sdp_stats}.  We see that finding $p_x$ and $p_x^{(M)}$  takes a fraction of the time it takes to solve \eqref{QP-Relax}. This shows that for large scale binary quadratic programs, \Cref{alg:Px} gives an efficient heuristic for finding feasible solutions close to the true global optimum. Also, we note that in none of our calculations \eqref{QP-Relax} was exact. In addition, determining the true global optimum for \eqref{QBP} is impractical for $n \geq 50$, so we are unable to evaluate which method gives an optimal solution closest to the true global optimum. These experiments and the results shown in this section motivate future work understanding the geometry and intersections of the cones $P_x$ and using this to inform efficient algorithms to find candidate solutions of \eqref{QBP}.

\subsection*{Acknowledgements} We thank Danielle Agostini for his insightful comments.
Research of Jose I. Rodriguez is supported by the Sloan Foundation and the Office of the Vice Chancellor for Research and Graduate Education at U.W. Madison with funding from the Wisconsin Alumni Research Foundation.

\bibliographystyle{plain}
\bibliography{refs.bib}

\begin{thebibliography}{10}

\bibitem{amendola2021solving}
Carlos {Améndola}, Julia {Lindberg}, and Jose~Israel {Rodriguez}.
\newblock Solving parameterized polynomial systems with decomposable
  projections.
\newblock {\em arXiv preprint arXiv:1612.08807}, 2021.

\bibitem{burer2020exact}
Samuel Burer and Yinyu Ye.
\newblock Exact semidefinite formulations for a class of (random and
  non-random) nonconvex quadratic programs.
\newblock {\em Math. Program.}, 181(1, Ser. A):1--17, 2020.

\bibitem{cifuentes2017on}
Diego Cifuentes, Sameer Agarwal, Pablo Parrilo, and Rekha Thomas.
\newblock On the local stability of semidefinite relaxations.
\newblock {\em Mathematical Programming}, 10 2017.

\bibitem{cifuentes2020the}
Diego Cifuentes, Corey Harris, and Bernd Sturmfels.
\newblock The geometry of {SDP}-exactness in quadratic optimization.
\newblock {\em Math. Program.}, 182(1-2, Ser. A):399--428, 2020.

\bibitem{flippo1996duality}
Olaf~E. Flippo and Benjamin Jansen.
\newblock Duality and sensitivity in nonconvex quadratic optimization over an
  ellipsoid.
\newblock {\em European Journal of Operational Research}, 94(1):167--178, 1996.

\bibitem{fujie1997semidefinite}
Tetsuya Fujie and Masakazu Kojima.
\newblock Semidefinite programming relaxation for nonconvex quadratic programs.
\newblock {\em J. Global Optim.}, 10(4):367--380, 1997.

\bibitem{goemans1995improved}
Michel~X. Goemans and David~P. Williamson.
\newblock Improved approximation algorithms for maximum cut and satisfiability
  problems using semidefinite programming.
\newblock {\em J. Assoc. Comput. Mach.}, 42(6):1115--1145, 1995.

\bibitem{khabbazibasmenj2014generalized}
Arash Khabbazibasmenj and Sergiy~A. Vorobyov.
\newblock Generalized quadratically constrained quadratic programming for
  signal processing.
\newblock In {\em 2014 IEEE International Conference on Acoustics, Speech and
  Signal Processing (ICASSP)}, pages 7629--7633, 2014.

\bibitem{kim2003exact}
Sunyoung Kim and Masakazu Kojima.
\newblock Exact solutions of some nonconvex quadratic optimization problems via
  {SDP} and {SOCP} relaxations.
\newblock {\em Comput. Optim. Appl.}, 26(2):143--154, 2003.

\bibitem{kochenberger2014the}
Gary Kochenberger, Jin-Kao Hao, Fred Glover, Mark Lewis, Zhipeng L\"{u}, Haibo
  Wang, and Yang Wang.
\newblock The unconstrained binary quadratic programming problem: a survey.
\newblock {\em J. Comb. Optim.}, 28(1):58--81, 2014.

\bibitem{lasserre2000global}
Jean~B. Lasserre.
\newblock Global optimization with polynomials and the problem of moments.
\newblock {\em SIAM J. Optim.}, 11(3):796--817, 2000/01.

\bibitem{laurent1995on}
Monique Laurent and Svatopluk Poljak.
\newblock On a positive semidefinite relaxation of the cut polytope.
\newblock volume 223/224, pages 439--461. 1995.
\newblock Special issue honoring Miroslav Fiedler and Vlastimil Pt\'{a}k.

\bibitem{lee2011mixed}
Jon Lee and Sven Leyffer.
\newblock {\em Mixed integer nonlinear programming}, volume 154.
\newblock Springer Science \& Business Media, 2011.

\bibitem{lindberg2021exploiting}
Julia {Lindberg}, Nigel {Boston}, and Bernard~C. {Lesieutre}.
\newblock Exploiting symmetry in the power flow equations using monodromy.
\newblock {\em ACM Communications in Computer Algebra}, 54(3):100--104, 2021.

\bibitem{LOCATELLI2015126}
Marco Locatelli.
\newblock Some results for quadratic problems with one or two quadratic
  constraints.
\newblock {\em Operations Research Letters}, 43(2):126--131, 2015.

\bibitem{luo2010semidefinite}
{Luo, Zhi-quan and Ma, Wing-kin and So, Anthony Man-Cho and Ye, Yinyu and
  Zhang, Shuzhong }.
\newblock Semidefinite relaxation of quadratic optimization problems.
\newblock {\em IEEE Signal Processing Magazine}, 27(3):20--34, 2010.

\bibitem{molzahn2019a}
Daniel~K. Molzahn and Ian~A. Hiskens.
\newblock A survey of relaxations and approximations of the power flow
  equations.
\newblock {\em Foundations and Trends in Electric Energy Systems},
  4(1-2):1--221, 2019.

\bibitem{nie2014optimality}
Jiawang Nie.
\newblock Optimality conditions and finite convergence of {L}asserre's
  hierarchy.
\newblock {\em Math. Program.}, 146(1-2, Ser. A):97--121, 2014.

\bibitem{papaspiliotopoulos2017a}
Vasileios~A. Papaspiliotopoulos, George~N. Korres, and Nicholas~G. Maratos.
\newblock A novel quadratically constrained quadratic programming method for
  optimal coordination of directional overcurrent relays.
\newblock {\em IEEE Transactions on Power Delivery}, 32(1):3--10, 2017.

\bibitem{parrilo2003semidefinite}
Pablo~A. Parrilo.
\newblock Semidefinite programming relaxations for semialgebraic problems.
\newblock volume~96, pages 293--320. 2003.
\newblock Algebraic and geometric methods in discrete optimization.

\bibitem{poljak1995a}
S.~Poljak, F.~Rendl, and H.~Wolkowicz.
\newblock A recipe for semidefinite relaxation for {$(0,1)$}-quadratic
  programming.
\newblock {\em J. Global Optim.}, 7(1):51--73, 1995.

\bibitem{riener2013exploiting}
Cordian Riener, Thorsten Theobald, Lina~Jansson Andr\'{e}n, and Jean~B.
  Lasserre.
\newblock Exploiting symmetries in {SDP}-relaxations for polynomial
  optimization.
\newblock {\em Math. Oper. Res.}, 38(1):122--141, 2013.

\bibitem{shor1987quadratic}
N.~Z. Shor.
\newblock Quadratic optimization problems.
\newblock {\em Izv. Akad. Nauk SSSR Tekhn. Kibernet.}, (1):128--139, 222, 1987.

\bibitem{sojoudi2014exactness}
Somayeh Sojoudi and Javad Lavaei.
\newblock Exactness of semidefinite relaxations for nonlinear optimization
  problems with underlying graph structure.
\newblock {\em SIAM J. Optim.}, 24(4):1746--1778, 2014.

\bibitem{tan2001the}
Peng~Hui Tan and L.K. Rasmussen.
\newblock The application of semidefinite programming for detection in {CDMA}.
\newblock {\em IEEE Journal on Selected Areas in Communications},
  19(8):1442--1449, 2001.

\bibitem{vavasis1990quadratic}
Stephen~A. Vavasis.
\newblock Quadratic programming is in {NP}.
\newblock {\em Inform. Process. Lett.}, 36(2):73--77, 1990.

\bibitem{wang2020on}
Alex~L. Wang and Fatma K{\i}l{\i}n{\c{c}}-Karzan.
\newblock On convex hulls of epigraphs of qcqps.
\newblock In Daniel Bienstock and Giacomo Zambelli, editors, {\em Integer
  Programming and Combinatorial Optimization}, pages 419--432, Cham, 2020.
  Springer International Publishing.

\bibitem{wang2021geometric}
Alex~L. Wang and Fatma Kilinc-Karzan.
\newblock A geometric view of sdp exactness in qcqps and its applications,
  2021.

\bibitem{wang2022tightness}
Alex~L Wang and Fatma K{\i}l{\i}n{\c{c}}-Karzan.
\newblock On the tightness of sdp relaxations of qcqps.
\newblock {\em Mathematical Programming}, 193(1):33--73, 2022.

\bibitem{wang2021tssos}
Jie Wang, Victor Magron, and Jean-Bernard Lasserre.
\newblock T{SSOS}: a moment-{SOS} hierarchy that exploits term sparsity.
\newblock {\em SIAM J. Optim.}, 31(1):30--58, 2021.

\bibitem{ugur2009convex}
Uğur Yildiran.
\newblock {Convex hull of two quadratic constraints is an LMI set}.
\newblock {\em IMA Journal of Mathematical Control and Information},
  26(4):417--450, 10 2009.

\bibitem{zhong2013dynamic}
Haiwang Zhong, Qing Xia, Yang Wang, and Chongqing Kang.
\newblock Dynamic economic dispatch considering transmission losses using
  quadratically constrained quadratic program method.
\newblock {\em IEEE Transactions on Power Systems}, 28(3):2232--2241, 2013.

\end{thebibliography}

\end{document}